\documentclass[draft]{amsart}

\usepackage{amssymb}
\usepackage{url}
\usepackage{amscd}
\usepackage{enumerate}
\usepackage[all]{xy}

\def \N{{\mathbb N}}

\def \R{{\mathbb R}}

\def \C{{\mathbb C}}

\def \1{{\mathbb 1}}

\theoremstyle{plain}
\newtheorem{theorem}{Theorem}
\newtheorem{proposition}{Proposition}
\newtheorem{definition}{Definition}
\newtheorem{lemma}{Lemma} 
\newtheorem{corollary}{Corollary}

\theoremstyle{remark}
\newtheorem{remark}{Remark}

\newtheorem{Exemp}{Example}

\begin{document}
\date{\today}
\title[Numerical radius]{Vector-valued numerical radius and $\sigma$-porosity.}
\author{Mohammed Bachir}

\address{Laboratoire SAMM 4543, Université Paris 1 Panthéon-Sorbonne, France.}
\email{Mohammed.Bachir@univ-paris1.fr}

\maketitle


\begin{abstract} It is well known that under certain conditions on a Banach space $X$, the set of bounded linear operators attaining their numerical radius is a dense subset. We prove in this paper that if $X$ is assumed to be uniformly convex and uniformly smooth then the set  of bounded linear operators attaining their numerical radius is not only a dense subset but also the complement of a $\sigma$-porous subset.  In fact, we generalize  the notion of numerical radius to a large class $\mathcal{Z}$ of vector-valued operators defined from $X\times X^*$ into a Banach space $W$  and we prove that the set of all elements of $\mathcal{Z}$ strongly (up to a symmetry) attaining their {\it numerical radius} is the complement of a $\sigma$-porous subset of $\mathcal{Z}$ and moreover  the {\it "numerical radius"} {\it Bishop-Phelps-Bollob\'as property}  is also satisfied for  this class. Our results extend  (up to the assumption on $X$) some known results in several directions: $(1)$ the density is replaced by being the complement of a  $\sigma$-porous subset, $(2)$ the operators attaining their {\it numerical radius} are replaced by operators strongly (up to a symmetry) attaining their {\it numerical radius} and $(3)$ the results are obtained in the vector-valued framework for general linear and non-linear vector-valued operators (including bilinear mappings and the classical space of bounded linear operators).
\end{abstract}


\newcommand\sfrac[2]{{#1/#2}}

\newcommand\cont{\operatorname{cont}}
\newcommand\diff{\operatorname{diff}}


{\bf 2020 Mathematics Subject Classiﬁcation:} 

\section{Introduction} 

In all this paper, $X$ denotes a real  Banach space, $S_X$, $B_X$ the unit sphere and the closed unit ball of $X$ respectively and $X^*$ the topological dual of $X$.  In this paper we will restrict ourselves to the case of real Banach spaces without treating the complex case in order to avoid technical difficulties.  However, to the best of our knowledge, there should not be a heavy obstacle to extend the results to the complex case. Recall that a bounded linear operator $T$ from $X$ into itself ($T \in \mathcal{L}(X)$) attains the numerical radius if there are $x_0\in S_X$, $x^*_0 \in S_{X^*}$ such that $x^*_0(x_0)=1$ and $|x^*_0(T(x_0))|=\sup\lbrace |x^*(T(x))|: (x,x^*)\in \Pi(X)\rbrace$, where $$\Pi(X)=\lbrace (x,x^*)\in S_X\times S_{X^*}, x^*(x)=1  \rbrace.$$ 

As in the case of operators attaining their norm, the question whether the set of operators attaining their numerical radius is dense in $\mathcal{L}(X)$ has attracted the attention of many authors. For instance, I. Berg and B. Sims  \cite{BS} proved that such denseness holds for uniformly convex Banach spaces.  C.S. Cardassi proved denseness for classical Banach spaces as $ c_0$, $C(K)$ and $L^1(\mu)$ as well as for uniformly smooth Banach spaces (see \cite{Car1, Car2, Car3}). There are many other studies on numerical radius attaining operators, for example the works by M.D. Acosta and R. Pay\'a or A. Capel, M. Mart\'in and J. Mer\'i (see \cite{Aco1, Aco2, Aco3, Aco4, CMM, ACKLM} and references therein).  For the study of the {\it Bishop-Phelps-Bollob\'as property} we refer to the work of M. D. Acosta, R. M. Aron, D. Garc\'ia, M. Maestre  in \cite{AAGM} (see also \cite{Aco}) and the work by J. Guirao and O. Kozhushkina  in \cite{GK} (see also \cite{KLM, AFS, GMMR}). More recent results, replacing the density of operators attaining their norm by being  the complement of  $\sigma$-porous was established in \cite{Ba} (see also the recent paper \cite{JMZ} for residuality).

\vskip5mm

The main result of this article is Theorem \ref{thmprin}, it will allow to give extensions both on the direction of operators attaining their numerical radius (Corollary \ref{Cor1}) and operators attaining their norm (Corollary \ref{CorA} and Corollary \ref{CorB}).

To the best of our knowledge, it is not known in the literature if we can replace the density of the set of operators attaining their numerical radius by being the complement of a $\sigma$-porous set. However, we will show in this paper that for real Banach spaces which are both uniformly convex and uniformly smooth, the density can effectively be replaced by  being the complement of a $\sigma$-porous set. In fact, we prove a more general result and obtain as a consequence the following corollary (Corollary ~\ref{Cor1} below). Let $X$  and $Y$ be Banach spaces. Let $\theta : S_X \times S_{Y^*} \to \R$ be a non-zero function  satisfying the following properties:

$(a)$ $\theta$ is uniformly continuous on $S_X \times S_{Y^*}$.

$(b)$ $\theta(-x, y^*)=\theta(x,-y^*)=-\theta(x,y^*)$, for all $(x,y^*)\in S_X \times S_{Y^*}$.

$(c)$ $\Pi_\theta(X,Y)\neq \emptyset$, where $$\Pi_\theta(X,Y)=\lbrace (x,y^*)\in S_X \times S_{Y^*}: \theta(x,y^*)=1\rbrace.$$

Let $W$ be any Banach space, by $(C_{b,e}(\Pi_\theta (X,Y), W), \|\cdot\|_{\infty})$ we denote the Banach space of all bounded continuous functions $\Phi$ from $\Pi_\theta (X,Y)$ into $W$ such that $\Phi(x,y^*)=\Phi(-x,-y^*)$ for all $(x,y^*)\in \Pi_\theta (X,Y)$. Define the generalized vector-valued {\it $\theta$-numerical radius} with respect to $\theta$ of an operator $\Phi \in C_{b,e}(\Pi_\theta (X,Y), W)$ in a natural way as follows:
$$v_\theta(\Phi):=\sup \lbrace \|\Phi(x,y^*)\|: (x,y^*)\in \Pi_\theta (X,Y)\rbrace:=\|\Phi\|_{\infty}.$$
We say that $\Phi$ attains strongly  its {\it "$\theta$-numerical radius"} if there exists some $(x_0,y^*_0)\in \Pi_\theta(X,Y)$ such that:

$(i)$ $ \|\Phi(x_0,y^*_0)\|=v_\theta(\Phi)$ and  

$(ii)$ $\min(\|x_n-x_0\|+\|y^*_n-y^*_0\|, \|x_n+x_0\|+\|y^*_n+y^*_0\|)\to 0$,  whenever $(x_n, y^*_n)\in \Pi_\theta(X,Y)$ and $\|\Phi(x_n,y^*_n)\|\to v_\theta (\Phi)=\|\Phi(x_0,y^*_0)\|$. 
\vskip5mm
We recover the classical numerical radius as a particular case by taking $Y=X$ and $\theta(x,x^*)=x^*(x)$ for every  $(x,x^*) \in S_X \times S_{X^*}$, in this case we have $\Pi_\theta(X,Y)=\Pi(X)$. By $[\mathcal{F}(X,Y)]\oplus W$, we denote the following space $$[\mathcal{F}(X,Y)]\oplus W:=\lbrace [T]\oplus w:  T\in \mathcal{F}(X,Y), w\in W\rbrace,$$
where $\mathcal{F}(X,Y)$ denotes the space of all finite rank bounded linear operators from $X$ into $Y$ and for each $T \in \mathcal{F}(X,Y)$ and each $ w\in W$, the map $[T]\oplus w$ is defined from $X\times Y^*$ into $W$ by $ (x,y^*)\mapsto y^*(T(x))w$.

We obtain immediately the following corollary by combining  Proposition \ref{corex} and the main result (Theorem \ref{thmprin}).

\begin{corollary} \label{Cor1} Let $X$  be a uniformly convex Banach space, $Y$ be a uniformly smooth Banach space and $\theta : S_X \times S_{Y^*} \to \R$ be a non-zero function  satisfying the properties $(a), (b)$ and $(c)$ (see above). Let $W$ be any Banach space and  $(\mathcal{Z}, \|\cdot\|_\mathcal{Z})$ be a Banach space included in $(C_{b,e}(\Pi_\theta(X,Y), W), \|\cdot\|_{\infty})$ and satisfying:

$(i)$ $\|\cdot\|_\mathcal{Z}\geq \|\cdot\|_{\infty}$.

$(ii)$ $\mathcal{Z}$ contains the space $[\mathcal{F}(X,Y)]\oplus W$ and is such that $\|[T]\oplus w\|_\mathcal{Z}\leq \|T\|\|w\|$ for all $(T,w)\in \mathcal{F}(X,Y)\times W$.

Then, for every $\Theta \in C_{b,e}(\Pi_\theta (X,Y), W)$ the set 
$$\mathcal{G}(\Theta):=\lbrace \Phi \in \mathcal{Z}: \Theta+\Phi \textnormal{  attains strongly its } \theta\textnormal{-numerical radius} \rbrace,$$
is the complement of a $\sigma$-porous subset of $(\mathcal{Z}, \|\cdot\|_{\mathcal{Z}})$. Moreover, we have the following version of the Bishop–Phelps–Bollobás theorem: given $\varepsilon >0$, there exists $\eta(\varepsilon)>0$ such that whenever $\Phi\in \mathcal{Z}$ with $v_\theta (\Phi)=1$ and $(x,y^*)\in \Pi_\theta (X,Y)$, satisfy 
$$\|\Phi(x,y^*)\|>1-\eta(\varepsilon),$$ 
there are $\Psi \in \mathcal{Z}$ with $v_\theta (\Psi)=1$ and $(z,t^*)\in \Pi_\theta (X,Y)$ such that 
$$\|\Psi(z,t^*)\|=1, \hspace{1mm} \|z-x\|+\|t^*-y^*\| < \varepsilon \textnormal{ and } v_\theta (\Psi-\Phi)<\varepsilon.$$
\end{corollary}

The class of Banach spaces $\mathcal{Z}$ satisfying the above corollary includes the whole space $(C_{b,e}(\Pi_\theta (X,Y), W), \|\cdot\|_{\infty})$, the space $\mathcal{B}(X\times Y^*, W)$ of bounded bilinear mappings $\Phi$ from $X\times Y^*$ into $W$ equipped with its natural norm $\|\Phi\|=\sup\lbrace \|\Phi(x,y^*)\|: x\in S_X, y^*\in S_{Y^*}\rbrace$, as well as (when $W=\R$) any closed subspace of $\mathcal{L}(X,Y)$ containing $\mathcal{F}(X,Y)$, where $\mathcal{L}(X,Y)$ (resp. $\mathcal{F}(X,Y)$) denotes the space of all bounded linear operators (resp. of all  finite rank linear operators) from $X$ into $Y$. As we have already said in the abstract, our results bring extensions and generalizations of results already known in the literature on several directions. For example: 

$(i)$ the notion of numerical radius is extended to the more general notion of {\it "$\theta$-numerical radius"},

$(ii)$ the density is replaced by being the complement of a  $\sigma$-porous subset, 

$(iii)$ the operators attaining their {\it numerical radius} are replaced by operators strongly (up to a symmetry) attaining their {\it $\theta$-numerical radius} and 

$(iv)$ the results are obtained in the vector-valued framework for general linear and non-linear vector-valued operators.

\vskip5mm
Our results  will be established by taking advantage from a  new variational principle (extending the  Deville-Godfroy-Zizler one's) established recently by the author in \cite[Theorem 4.6]{Ba}.

\section{Notion and Definitions} 
We first recall some notions and definitions that we are going to use in this paper.  
\begin{definition} Let $C$ be a nonempty set and $\Gamma : C\times C\to \R^+$. We say that $\Gamma$ is a pseudometric if 
	
	$(1)$ $\gamma(x,x)= 0$, for all $x \in C$.
	
	$(2)$ $\gamma(x,y)=\gamma(y,x)$, for all $x \in C$.
	
	$(3)$ $\gamma(x,y)\leq \gamma(x,z)+ \gamma(z,y)$, for all $x,y,z \in C$.
\end{definition}
Unlike a metric space,  one may have $\gamma ( x , y ) = 0$  for distinct values $ x \neq y $. A pseudometric induces an equivalence relation, that converts the pseudometric space into a metric space. This is done by defining $x\sim y$ if  $\gamma ( x , y ) = 0$.
Let $\Gamma_\gamma:  C\to  C/\sim $ the canonical surjection mapping and let 
$$d_\gamma(\Gamma_\gamma(x),\Gamma_\gamma(y)):=\gamma(x,y).$$
Then, $(C/\sim,d_\gamma)$ is a well defined metric space. We say that $(C,\Gamma)$ is a complete pseudometric space, if $(C/\sim,d_\gamma)$ is a complete metric space.

We  need the following definition of uniform separation property for pseudometrics introduced in \cite{Ba}.
\begin{definition} \label{USP}
	Let $X$ be a Banach space, $C$ be a subset of the dual $X^*$ and $(C,\gamma)$ be a pseudometric space.  We say that  $(C,\gamma)$ has the weak$^*$-uniform separation property (in short $w^*USP$) in $X^*$ if  for every $\varepsilon>0$ small enough, there exists $\varpi_C(\varepsilon)>0$ such that for every $p\in C$, there exists $x_{p,\varepsilon}\in B_X$ satisfying: $\forall q\in C,$
	\begin{eqnarray*}
		\gamma(p,q)\geq \varepsilon \Longrightarrow \langle p , x_{p,\varepsilon} \rangle -  \varpi_C(\varepsilon) \geq  \langle q , x_{p,\varepsilon} \rangle.
	\end{eqnarray*} 
The function $\varpi_C$ is called, the {\it modulus of uniform separation} of $(C,\gamma)$. 
\end{definition}
\begin{Exemp} Let $X$ be a Banach space  whose dual $X^*$ is uniformly convex. Then, $(S_{X^*},\|\cdot\|)$ has the $w^*USP$ in $X^*$ (see \cite{Ba}).
\end{Exemp}
We recall the notion of $\sigma$-porosity. In the following definition,  $\mathring{B}_X(x ; r)$ stands for the open ball in $X$ centered at $x$ and with radius $r > 0$.
\begin{definition}\label{prous}  Let $(X ; d)$ be a metric space and $A$ be a subset of $X$. The set $A$ is
said to be porous in $X$ if there exist $\lambda_0 \in (0; 1]$ and $r_0 > 0$ such that for any $x \in X$
and $r \in (0; r_0]$ there exists $y \in X$ such that $\mathring{B}_X(y; \lambda_0r) \subseteq \mathring{B}_X(x; r) \cap (X \setminus A)$. The set
$A$ is called $\sigma$-porous in $X$ if it can be represented as a countable union of porous sets in $X$.
\end{definition}
 Every $\sigma$-porous set is of first Baire category. Moreover, in $\R^n$, every $\sigma$-porous set is of Lebesque measure zero. However,  there does exist a non-$\sigma$-porous subset of $\R^n$ which is of the first category and of Lebesgue measure zero. For more informations about $\sigma$-porosity, we refer to \cite{Za}. 


\section{The main result and applications} \label{SP}
The first step of our work, consists in building an adequate formalism, using the Hausdorff distance and quotient spaces, in order to be able to use judiciously a result recently obtained by the author in  \cite[Theorem 4.6]{Ba}. We will also need to establish some lemmas and propositions.

Recall that the Hausdorff distance $d_H(A,B)$ between two nonempty  closed and bounded subsets $A, B$ of some metric space $(E,d)$ is defined as follows.
\begin{eqnarray*}
 d_H ( A , B ) &:=& \max (\sup_{ y\in B}  d(y , A ) , \sup_{x \in A} d ( x , B ) ) \\
&=& \max \lbrace \sup_{y \in B} \inf_{x \in A} d ( x , y ) , \sup_{x\in A} \inf_{y \in B} d ( x , y ) \rbrace.
\end{eqnarray*}
It is well-know that if $(E,d)$ is complete and if we denote $K(E)$ the set of all nonempty compact of $E$, then $(K(E), d_H)$ is a complete metric space. We are interested in this paper on compact subsets consisting only on two points. 

\vskip5mm
Let $V$ be a Banach space. We define an equivalence relation $\sim$ on $V$ as follows: $x \sim y$ if and only if $x=y$ or $x=-y$. Let $\Gamma_V: V \to  V/\sim $ the canonical surjection mapping, that is, $\Gamma_V(x)=\lbrace x, -x \rbrace$ for all $x \in V$. We equip $\Gamma_V(V)=V/\sim$ with the Hausdorff distance $d_H$. We see easily that: $\forall x, y \in V$,
$$d_H(\Gamma_V(x),\Gamma_V(y))=\min(\|x-y\|,\|x+y\|).$$
Let $W$ be a Banach space and $L$ be a symmetric closed subset of $V$. By $(C_{b,e}(L, W), \|\cdot\|_{\infty})$ we denote the space of all $W$-valued bounded continuous even function $\Phi : (L, \|\cdot\|) \to W$, that is, $\Phi$ is a bounded continuous function satisfying  $\Phi(x)=\Phi(-x)$  for all $x\in L$, equipped with the sup-norm. By $(C_b(\Gamma_V(L), W), \|\cdot\|_{\infty})$ we denote the space of all $W$-valued bounded continuous functions on $(\Gamma_V(L), d_H)$ equipped with the sup-norm.
We give the following elementary proposition.
\begin{proposition} \label{P1} Let $V , W$ be Banach spaces  and $L$ be a symmetric closed subset of $V$. Then, the following assertions hold.

$(i)$ $(\Gamma_V(L), d_H)$ is a complete metric space.  

$(ii)$ For every $\Phi \in C_{b,e}(L, W)$, there exists a unique continuous function $\tilde{\Phi}: (\Gamma_V(L), d_H) \to W$ such that $\Phi=\tilde{\Phi}\circ \Gamma_V$. Hence, the map $I_L: (C_{b,e}(L, W), \|\cdot\|_{\infty})\to (\widetilde{C}_{b,e}(\Gamma_V(L), W), \|\cdot\|_{\infty})$ defined by $I_L(\Phi)= \tilde{\Phi}$ is a linear surjective isometry, where $\widetilde{C}_{b,e}(\Gamma_V(L), W):=I_L(C_{b,e}(L,W))$. On the other hand, $\widetilde{C}_{b,e}(\Gamma_V(L), W)$ is a closed subspace of $(C_b(\Gamma_V(L), W), \|\cdot\|_{\infty})$.

\end{proposition}
\begin{proof} The part $(i)$ can be easily proven directely or  by observing that on one hand $\Gamma_V(L)\subseteq K(L)$ ($K(L)$ is the set of all compact subsets of $L$ which is complete for Hausdorff distance $d_H$) and on the other hand, $\Gamma_V(L)$ is closed for the Hausdorff distance. To see part $(ii)$, it suffices to use the factorization lemma. Indeed, let $\Phi \in C_{b,e}(L, W)$. If $\Gamma_V(x)=\Gamma_V(y)$, then $x=y$ or $x=-y$. Since $\Phi$ is an even function we have $\Phi(x)=\Phi(-x)=\Phi(y)$. Thus, by the factorization lemma there exists a unique continuous mapping  $\tilde{\Phi}: (\Gamma_V(L), d_H) \to W$ such that $\Phi=\tilde{\Phi}\circ \Gamma_V$. Now, it is clear that the map $I_L: \Phi \to \tilde{\Phi}$ is a linear surjective isometry, for the sup-norm,  into its image. On the other hand, it is trivial that $\widetilde{C}_{b,e}(\Gamma_V(L), W)$ is a closed subspace of $C_b(\Gamma_V(L), W)$.
\end{proof}
\begin{definition} Let $V,W$ be Banach spaces and $L$ be a symmetric closed subset of $V$. Let $\Phi \in C_{b,e}(L, W)$. We say that $\Phi$ attains $d_H$-strongly its maximum on $L$ if there exists $x_0\in L$ such that $$\sup_{x\in L}\|\Phi(x)\|=\|\Phi(x_0)\|$$ and $d_H(\Gamma_V(x_n), \Gamma_V(x_0)))\to 0$, whenever $\|\Phi(x_n)\|\to \sup_{x\in L}\|\Phi(x)\|$. In other words, $\Phi$ attains $d_H$-strongly its maximum on $L$ at $x_0\in L$ if and only if the function $\|\widetilde{\Phi}
(\cdot)\|:=\|I_L(\Phi)(\cdot)\|$ attains its strong maximum (in the classical way) on $(\Gamma_V(L), d_H)$ at $\Gamma_V(x_0)$.
\end{definition}

Let $V$ be a Banach space and $L$ be a nonempty symmetric and closed subset of $V$. Let $(\mathcal{Z}, \|\cdot\|_\mathcal{Z})$ be a Banach space included in $C_{b,e}(L, W)$ and such that  $\|\cdot\|_\mathcal{Z}\geq \|\cdot\|_{\infty}$ and denote $\widetilde{\mathcal{Z}}=I_L(\mathcal{Z})$, which is included in $\widetilde{C}_{b,e}(\Gamma_V(L), W)\subseteq C_b(\Gamma_V(L),W)$. We equipp $\widetilde{\mathcal{Z}}$ with the norm defined by $\|\widetilde{\Phi}\|_{\widetilde{\mathcal{Z}}} := \|I_L(\Phi)\|_{\widetilde{\mathcal{Z}}}:=\|\Phi\|_\mathcal{Z}$ for all $\Phi \in \mathcal{Z}$, so that $\mathcal{Z}$ and $\widetilde{\mathcal{Z}}$ are isometrically isomorphic by the isomorphism $I_L$. Define the subset $\mathcal{C}_L$ of the dual $\widetilde{\mathcal{Z}}^*$ (thanks to the fact that $\|\cdot\|_{\mathcal{Z}}\geq \|\cdot\|_{\infty}$) as follows:

\begin{eqnarray} \label{F1}
\mathcal{C}_L:=\lbrace w^*\circ \delta_{\Gamma_V (x)}: x \in L, w^*\in S_{W^*}\rbrace,
\end{eqnarray}
where, for each fixed $x\in L$, $w^* \in W^*$, the linear and continuous functional $ w^*\circ \delta_{\Gamma_V(x)}: \widetilde{\mathcal{Z}} \to \R$ is defined by $$w^*\circ \delta_{\Gamma_V (x)}(\tilde{\Phi}):=w^*(\tilde{\Phi}\circ \Gamma_V (x)):=w^*(\Phi(x)), \textnormal{ for all } \widetilde{\Phi}\in \widetilde{\mathcal{Z}}.$$ We define $\gamma_\mathcal{P}: \mathcal{C}_L \times \mathcal{C}_L \to \R^+$   as follows: for all $w^*, v^*\in S_{W^*}$ and for all $x; y \in L$,
\begin{eqnarray} \label{F2}
\gamma_\mathcal{P}(w^*\circ \delta_{\Gamma_V (x)}, v^*\circ \delta_{\Gamma_V (y)}):=d_H(\Gamma_V (x), \Gamma_V (y)).
\end{eqnarray}
Assume for the moment that $(\mathcal{C}_L, \gamma_\mathcal{P})$ is a well defined pseudometric subset of $\widetilde{\mathcal{Z}}^*$ (see for instance  Lemma \ref{lemC0}).
\vskip5mm
 We give the following version of \cite[Theorem 4.6]{Ba} adapted to the context of this article and which will be of great use later.

\begin{proposition} \label{PP1} $($\textnormal{see} \cite[Theorem 4.6]{Ba}$)$ Let $V, W$ be Banach spaces and $L$ be a symmetric closed subset of $V$. Let $(\mathcal{Z}, \|\cdot\|_\mathcal{Z})$ be a Banach space included in $(C_{b,e}(L, W), \|\cdot\|_{\infty})$ and such that $\|\cdot\|_\mathcal{Z}\geq \|\cdot\|_{\infty}$. Assume that $(\mathcal{C}_L, \gamma_\mathcal{P})$ is a well defined complete pseudometric space having the $w^*USP$ in the dual $\widetilde{\mathcal{Z}}^*$. Then, for every $\Theta \in C_{b,e}(L, W)$, the set 
$$ \mathcal{G}(\Theta):=\lbrace \Phi \in \mathcal{Z}: \Theta + \Phi \textnormal{  attains } d_H\textnormal{-strongly its maximum on } L\rbrace,$$
is the complement of a $\sigma$-porous subset of $(\mathcal{Z}, \|\cdot\|_{\mathcal{Z}})$. Moreover, we have the following version of the Bishop–Phelps–Bollobás theorem: given $\varepsilon >0$, there exists $\eta(\varepsilon)>0$ such that whenever $\Phi\in \mathcal{Z}$ with $\|\Phi\|_{\infty}=1$ and $x\in L$, satisfy 
$$\|\Phi(x)\|>1-\eta(\varepsilon),$$ 
there are $\Psi \in \mathcal{Z}$ with $\|\Psi\|_{\infty}=1$ and $y\in L$ such that 
$$\|\Psi(y)\|=1, \hspace{1mm} \|x-y\| < \varepsilon \textnormal{ and } \|\Psi-\Phi\|_{\infty}<\varepsilon.$$
\end{proposition}
\begin{proof} We apply directely \cite[Theorem 4.6]{Ba} to the Banach space $(\widetilde{\mathcal{Z}}, \|\cdot\|_{\widetilde{\mathcal{Z}}})$ which is included in $C_b(\Gamma_V(L), W)$ (where, $\Gamma_V(L)$ is equipped with the complete metric $d_H$), using the fact that $\mathcal{Z}$ and $\widetilde{\mathcal{Z}}$ are isometrically isomorphic by the isometry $I_L$ and observing that by the definitions a mapping $\widetilde{\Phi}\in \widetilde{C}_{b,e}(\Gamma_V(L), W)$ strongly attains its sup-norm if and only if $\Phi \in C_{b,e}(L, W)$ attains  $d_H$-strongly its maximum on $L$. Notice also that for every $\Phi \in \mathcal{Z}$, if there exists $y\in L$ such that 
$\|\Phi(y)\|=1$ and  $d_H(\Gamma_V(x), \Gamma_V(y)) <\varepsilon$, then there exists $\bar{y}=\pm y\in L$ such that $\|\Phi(\bar{y})\|=1$ and $\|x-\bar{y}\|< \varepsilon$ (since, $d_H(\Gamma_V(x), \Gamma_V(y))=\min(\|x-y\|, \|x+y\|)$ and $\Phi$ is an even function).
\end{proof}
The interest now is to apply the previous proposition with judiciously chosen sets $L$, when the Banach space $V$  satisfies some reasonable property. The case of  $V=X\times Y^*$ (where $X,Y$ are Banach spaces),  $L\subseteq S_X\times S_{Y^*}$ and $(\mathcal{Z}, \|\cdot\|_\mathcal{Z})$ is a Banach space included in $(C_{b,e}(L, W),\|\cdot\|_{\infty})$, will allow to obtain extensions and new results around the problem of operators attaining  their  numerical radius (resp. the problem of operators attaining  their norm). In particular, Corollary \ref{Cor1} in the introduction will be given with  $L=\Pi_\theta (X,Y)$. Thus, in what follows, we will give a general sufficient condition on the spaces $X, Y, L$ and the Banach space  $\mathcal{Z}$  such that $(\mathcal{C}_L, \gamma_\mathcal{P})$ will be a complete pseudometric space having the $w^*USP$ in the dual $\widetilde{\mathcal{Z}}^*$.  To do this, we need some definitions and lemmas. 

\subsection{The property $(P)$} 
We need the following definition which is closely related to the notion of $w^*USP$ in Definition \ref{USP}. The space $X\times Y^*$ will be equipped with the norm $\|(x, y^*)\|:=\|x\|+\|y^*\|$.  In what follows, we will simply denote $\Gamma$ to designate $\Gamma_{X\times Y^*}$. Recall that for all $(x,y^*), (z,t^*) \in X\times Y^*$,
$$d_H(\Gamma(x,y^*),\Gamma(z,t^*))=\min(\|x-z\|+\|y^*-t^*\|,\|x+z\|+\|y^*+t^*\|).$$
\begin{definition} \label{(P)} Let $X,Y$ be Banach spaces and $L$ be a nonempty symmetric  and closed subset of $S_X\times S_{Y^*}$. We say that $L$ has property $(P)$ in $X\times Y^*$ if for every $\varepsilon >0$ small enough, there exists a modulus $\varpi (\varepsilon)>0$ such that for every $(x,y^*) \in L$, there exists $(x^*_{\varepsilon},y_{\varepsilon})\in  B_{X^*}\times B_Y$ satisfying: $\forall (z,t^*)\in L$
\begin{eqnarray*} \label{U11}
		d_H(\Gamma(x,y^*), \Gamma(z,t^*))  \geq \varepsilon \Longrightarrow x^*_{\varepsilon}(x) +y^*(y_\varepsilon)-  \varpi (\varepsilon) \geq  |x^*_{\varepsilon} (z)|+|t^*(y_\varepsilon)|.
	\end{eqnarray*}
\end{definition}

Recall that a Banach space $(X,\|.\|)$ is uniformly convex if for each $\varepsilon\in (0,2]$,
\begin{eqnarray*}
\delta(\varepsilon):=\inf\lbrace 1-\|\frac{x+y}{2}\|: x,y\in S_X; \|x-y\|\geq \varepsilon\rbrace >0.
\end{eqnarray*}
We give the following lemma which will allow us to say that our results on the subject of operators attaining their {\it $\theta$-numerical radius} applies in particular to spaces which are uniformly convex and uniformly smooth. More concrete situations will be provided in Proposition \ref{corex}.

\begin{lemma} \label{Prop3} Let $X$ be a uniformly convex Banach space and $Y$ be a uniformly smooth Banach space. Let $L$ be a nonempty closed and symmetric subset of $S_X \times S_{Y^*}$ satisfying: for all $(x,y^*), (z,t^*)\in L$ and for every $\varepsilon >0$ small enough,
$$ d_H(\Gamma(x,y^*), \Gamma(z,t^*))  \geq \varepsilon \Longrightarrow d_H(\Gamma(-x,y^*), \Gamma(z,t^*))  \geq \varepsilon.$$
Then, $L$ has the property $(P)$.
\end{lemma}

\begin{proof} Let $\delta_X(\cdot)$ and $\delta_{Y^*}(\cdot)$ be the modulus of uniform convexity of $X$ and $Y^*$ respectively ($Y^*$ is uniformly convex since $Y$ is a smooth Banach space). Let us denote
$$\varpi (\varepsilon):=\frac{1}{3}\min(\delta_X(\frac{\varepsilon}{2}), \delta_{Y^*}(\frac{\varepsilon}{2})),$$
and for every $(x,y^*) \in L$, let $x^*_{x}\in S_{X^*}$ be such that $x^*_{x}(x)=1$ and $y_{y^*,\varepsilon}\in S_Y$ be such that 
$$y^*(y_{y^*,\varepsilon})>1-\varpi (\varepsilon).$$
Now, let $(z, t^*)\in L$ be such that $d_H(\Gamma(x,y^*),\Gamma(z,t^*))\geq \varepsilon$. We are going to prove that 
\begin{eqnarray} \label{Pr}
|x^*_{x}(z)|+ |t^*(y_{y^*,\varepsilon})| &\leq& x^*_{x}(x) + y^*(y_{y^*,\varepsilon}) -\varpi (\varepsilon),
\end{eqnarray}
which gives  the property $(P)$. Indeed, since $d_H(\Gamma(x,y^*),\Gamma(z,t^*))\geq \varepsilon$ then by assymption we also have $d_H(\Gamma(-x,y^*),\Gamma(z,t^*))\geq \varepsilon$, thus we have
\begin{eqnarray} \label{Mr}
\|x-z\|+\|y^*-t^*\|\geq \varepsilon \textnormal{ and } \|x+z\|+\|y^*+t^*\|\geq \varepsilon.
\end{eqnarray}
\begin{eqnarray} \label{Mr1}
\|x+z\|+\|y^*-t^*\|\geq \varepsilon \textnormal{ and } \|x-z\|+\|y^*+t^*\|\geq \varepsilon.
\end{eqnarray}
We need to treat two cases.

{\bf Case 1:} Suppose that $\|x-z\|<\frac{\varepsilon}{2}$ or  $\|x+z\|<\frac{\varepsilon}{2}$. Then, using $(\ref{Mr})$ and $(\ref{Mr1})$, we get that  $\|y^*-t^*\|\geq \frac{\varepsilon}{2}$ and $\|y^*+t^*\| \geq \frac{\varepsilon}{2}$. Thus, by the uniform convexity of $Y^*$, we have  for all $y\in S_{Y}$
\begin{eqnarray*} 
\langle \frac{y^*+t^*}{2}, y \rangle \leq \|\frac{y^*+t^*}{2}\| \leq 1- \delta_{Y^*}(\frac{\varepsilon}{2}).
\end{eqnarray*}
\begin{eqnarray*} 
\langle \frac{y^*-t^*}{2}, y \rangle \leq \|\frac{y^*-t^*}{2}\| \leq 1- \delta_{Y^*}(\frac{\varepsilon}{2}).
\end{eqnarray*}
By using the fact that $y^*(y_{y^*,\varepsilon})>1-\varpi (\varepsilon)$ and applying the above inequalities with $y=y_{y^*,\varepsilon}$, we get 
\begin{eqnarray*} 
t^* (y_{y^*,\varepsilon}) \leq y^*(y_{y^*,\varepsilon})- \varpi (\varepsilon),
\end{eqnarray*}
\begin{eqnarray*} 
- t^* (y_{y^*,\varepsilon})  \leq y^*(y_{y^*,\varepsilon}) - \varpi (\varepsilon) .
\end{eqnarray*}
Equivalently, we have
\begin{eqnarray*} 
|t^* (y_{y^*,\varepsilon})| \leq y^*(y_{y^*,\varepsilon})- \varpi (\varepsilon).
\end{eqnarray*}
On the other hand, we have $|x^*_x (z)| \leq 1=x^*_x(x)$. Thus, 
\begin{eqnarray*} 
|x^*_x (z)| + |t^* (y_{y^*,\varepsilon})| &\leq& x^*_x(x)+y^*(y_{y^*,\varepsilon})-\varpi (\varepsilon).
\end{eqnarray*}
This gives $(\ref{Pr})$.

{\bf Case 2:} Suppose that $\|x-z\| \geq \frac{\varepsilon}{2}$ and $\|x+z\| \geq \frac{\varepsilon}{2}$. By the uniform convexity of $X$, we have for all $x^* \in S_{X^*}$
\begin{eqnarray*} 
\langle x^*, \frac{x+z}{2}\rangle \leq \|\frac{x+z}{2}\| \leq 1- \delta_X(\frac{\varepsilon}{2}).
\end{eqnarray*}
\begin{eqnarray*} 
\langle x^*, \frac{x-z}{2}\rangle \leq \|\frac{x-z}{2}\| \leq 1- \delta_X(\frac{\varepsilon}{2}).
\end{eqnarray*}
Thus, combining the above two inequalities, we see that for all $x^* \in S_{X^*}$ such that $x^*(x)=1$ we have
\begin{eqnarray*} 
|x^*(z)| \leq  1- 2\delta_X(\frac{\varepsilon}{2})=x^*(x)-2\delta_X(\frac{\varepsilon}{2}).
\end{eqnarray*}
On the other hand, using again the fact that $y^*(y_{y^*,\varepsilon})>1-\varpi (\varepsilon)$, we have
\begin{eqnarray*}
|t^*(y_{y^*,\varepsilon})|\leq 1< y^*(y_{y^*,\varepsilon}) +\varpi (\varepsilon)
\end{eqnarray*}
Thus, we have (with $x^*=x^*_{x}$)
\begin{eqnarray*} 
|x^*_{x}(z)|+ |t^*(y_{y^*,\varepsilon})| &\leq&  x^*_{x}(x) + y^*(y_{y^*,\varepsilon}) +\varpi (\varepsilon) -2\delta_X(\frac{\varepsilon}{2})  \\
&\leq& x^*_{x}(x) + y^*(y_{y^*,\varepsilon}) -\varpi (\varepsilon)
\end{eqnarray*}
This gives $(\ref{Pr})$.
\end{proof}

We give in the following proposition a general but more concrete examples of nonempty closed and symmetric subsets having the property $(P)$. 

\begin{proposition}  \label{corex} 
Let $X$ and $Y$ be Banach spaces and $\theta : S_X \times S_{Y^*} \to \R$ be a non-zero function  satisfying the following properties:

$(i)$ $\theta$ is uniformly continuous on $S_X \times S_{Y^*}$ with a modulus of uniform continuity $\alpha(\cdot)$.

$(ii)$ $\theta(-x, y^*)=\theta(x,-y^*)=-\theta(x,y^*)$, for all $(x,y^*)\in S_X \times S_{Y^*}$.

Then, for every $s>0$ such that $\theta^{-1}([s,+\infty[)\neq \emptyset$ and for every nonempty  subset $L$ of $\theta^{-1}([s,+\infty[)$, we have that: $\forall (x,y^*), (z,t^*)\in L$, 
$$d_H(\Gamma(-x,y^*), \Gamma(z,t^*))  \geq \alpha(s).$$
Consequently (see Lemma \ref{Prop3}), if we assume that $X$ is uniformly convex and $Y$ is uniformly smooth, then every nonempty closed and symetric subset $L$ of $\theta^{-1}([s,+\infty[)$ has the property $(P)$. In particular, the closed and symmetric set $\theta^{-1}([s,+\infty[)$ as well the set $\theta^{-1}(\lbrace  s\rbrace)$ have the property $(P)$, whenever $\theta^{-1}(\lbrace  s\rbrace)\neq \emptyset$.
\end{proposition}

\begin{proof} Let $\alpha(\cdot)$ be a modulus of uniform continuity of $\theta$, that is, for every $r>0$ and every $(x_1,y^*_1), (x_2,y^*_2) \in S_X \times S_{Y^*}$
\begin{eqnarray*} 
\|x_1-x_2\|+\|y^*_1-y^*_2\| < \alpha(r) \Longrightarrow |\theta(x_2,y^*_2)-\theta(x_1,y^*_1)|<r.
\end{eqnarray*}
Let $s>0$ be such that $\theta^{-1}([s,+\infty[)\neq \emptyset$ and let $L$ be any nonempty closed and  symmetric subset of $\theta^{-1}([s,+\infty[)$. We prove that for every $(x,y^*), (z,t^*)\in L$, we have  $d_H(\Gamma(-x,y^*), \Gamma(z,t^*))  \geq \alpha(s)$. Indeed, suppose by contradiction that 
$$\|x+z\|+ \|y^*-t^*\| < \alpha(s).$$
Then, by the assymption $(ii)$, the  uniform continuity of $\theta$ and the definition of $L$, we get
\begin{eqnarray*} 
2s\leq \theta(z,t^*)+\theta(x,y^*)=\theta(z,t^*)-\theta(-x,y^*)<s,
\end{eqnarray*}
wich is a contradiction. Thus, we have $\|x+z\|+ \|y^*-t^*\| \geq \alpha(s)$. In a similar way, we prove that $ \|x-z\|+\|y^*+t^*\| \geq \alpha(s)$. Thus, $d_H(\Gamma(-x,y^*), \Gamma(z,t^*))  \geq \alpha(s)$. It follows that the set $L$ satisfies the property of Lemma \ref{Prop3} for every $\varepsilon \in (0, \alpha(s))$. Hence,  $L$ has the property $(P)$ whenever $X$ is uniformly convex and $Y$ is smooth by Lemma \ref{Prop3}.
\end{proof}
\begin{remark} \label{Pix}
Notice that as an immediate consequence of the above proposition we get that if $X$ is uniformly convex, $Y$ is uniformly smooth and $G: S_X \to S_Y$ is a uniformly continuous mapping satisfying $G(-x)=-G(x)$ for all $x\in S_X$, then the map defined by $\theta_G: (x,y^*)\mapsto y^*(G(x))$ for all $(x,y^*)\in S_X\times S_{Y^*}$, satisfies the hypothesis of Proposition \ref{corex}. It follows that the nonempty closed and symmetric set $$\Pi_G (X,Y)=\lbrace (x,y^*)\in S_X\times S_{Y^*}:  y^*(G(x))=1  \rbrace=\theta^{-1}_G(\lbrace 1\rbrace)\neq \emptyset,$$ has the property $(P)$. Notice that we recover the classical set $\Pi(X)$ by taking $G=I$ the identity mapping. We refeer to \cite{KMMPQ, MMPa} for the study of {\it "numerical range"} with respect some fixed bounded linear operator $G: X \to Y$ such that $\|G\|=1$, where the set 
$$\theta^{-1}_{G}([1- \delta,+\infty[)=\lbrace (x,y^*)\in S_X \times S_{Y^*}: y^*(G x)\geq 1-\delta \rbrace\neq \emptyset,$$
for $0< \delta <1$, is considered. Notice, that from Proposition \ref{corex}, the set $\theta^{-1}_{G}([1- \delta,+\infty[)$ has also property $(P)$.
\end{remark}

\subsection{The main result}
Let $X,Y, W$ be Banach spaces. We denote $\mathcal{F}(X,Y)$ the space of all finite rank bounded linear operators from $X$ into $Y$ and 
$$[\mathcal{F}(X, Y)]\oplus W:=\lbrace [T]\oplus w:  T\in \mathcal{F}(X, Y),  w\in W\rbrace,$$
where for each $T \in \mathcal{F}(X, Y)$ and each $ w\in W$, the map $[T]\oplus w$ is defined from $X\times Y^*$ into $W$ by $ (x,y^*)\mapsto y^*(T(x))w$.
\begin{lemma} \label{lemC0} Let $X, Y, W$  be Banach spaces  and $L$ be a symmetric  and closed subset of $S_X\times S_{Y^*}$ having property $(P)$. Let $(\mathcal{Z}, \|\cdot\|_\mathcal{Z})$ be a Banach space included in $C_{b,e}(L, W)$, such that:

$(i)$ $\|\cdot\|_\mathcal{Z}\geq \|\cdot\|_{\infty}$.

$(ii)$ $\mathcal{Z}$ contains the space $[\mathcal{F}(X, Y)]\oplus W$.

 Then, the map $\gamma_\mathcal{P}$ is well defined and $(\mathcal{C}_L,\gamma_\mathcal{P})$ (see the formulas $(\ref{F1})$ and $(\ref{F2})$) is a complete pseudometric subset of $\widetilde{\mathcal{Z}}^*:=(I_L(\mathcal{Z}))^*$.
\end{lemma}
\begin{proof} Notice that from the fact that $\|\cdot\|_\mathcal{Z}\geq \|\cdot\|_{\infty}$, we get that the evaluations $w^*\circ \delta_{\Gamma(x,y^*)}$ are linear continuous from $\widetilde{\mathcal{Z}}$ into $\R$ and so $\mathcal{C}_L$ is a subset of the dual $\widetilde{\mathcal{Z}}^*$. To see that $\gamma_\mathcal{P}$ is well defined, it suffices to prove that: if $w^*\circ \delta_{\Gamma(x,y^*)}= v^*\circ \delta_{\Gamma(z,t^*)}$ on $\widetilde{\mathcal{Z}}=I_L(\mathcal{Z})$, then $\Gamma(x,y^*) =\Gamma(z,t^*)$. Indeed, suppose that $w^*\circ \delta_{\Gamma(x,y^*)}= v^*\circ \delta_{\Gamma(z,t^*)}$ on $\widetilde{\mathcal{Z}}=I_L(\mathcal{Z})$. In particular, for all $T\in \mathcal{F}(X,Y)$ and all $e\in S_W$, $$w^*\circ \delta_{\Gamma(x,y^*)}(I_L([T]\oplus e))= v^*\circ \delta_{\Gamma(z,t^*)}(I_L([T]\oplus e)).$$ 
This gives,  
\begin{eqnarray} \label{R1}
y^*(T(x))w^*(e)=t^*(T(z))v^*(e),
\end{eqnarray}
for all $T\in \mathcal{F}(X,Y)$ and all $e\in S_W$. Let $x^*\in X^*, y_\alpha\in Y$ be such that $\|x^*\|=x^*(x)=1$ and $\|y_\alpha\|=1$, $y^*(y_\alpha)\geq 1-\alpha$ for every $\alpha\in (0,1)$. Thus, with  $T_{x,y^*,\alpha}=x^*(\cdot)y_\alpha \in \mathcal{F}(X,Y)$, we get that 
\begin{eqnarray} \label{RR}
w^*(e)y^*(y_\alpha)=t^*(y_\alpha)x^*(z)v^*(e),
\end{eqnarray}
 for every $\alpha \in (0,1)$ and so we see that $|w^*(e)|\leq |v^*(e)|$ for all $e\in S_W$. In a similar way,  we get $|v^*(e)|\leq |w^*(e)|$ for all $e\in S_W$. Thus,  $|w^*(e)|= |v^*(e)|$ for all $e\in S_W$. Using $(\ref{R1})$, we obtain for all $T\in \mathcal{F}(X,Y)$,
\begin{eqnarray} \label{R2}
|y^*(T(x))|=|t^*(T(z))|.
\end{eqnarray}
Suppose by contradiction that $\Gamma(x,y^*) \neq \Gamma(z,t^*)$, then for some $\varepsilon >0$, we have that $d_H(\Gamma(x,y^*), \Gamma(z,t^*))  \geq \varepsilon$. Using property $(P)$, there exists $(x^*_{\varepsilon},y_{\varepsilon})\in  B_{X^*}\times B_Y$  such that 
$$x^*_{\varepsilon}(x) +y^*(y_\varepsilon)-  \varpi (\varepsilon) \geq  |x^*_{\varepsilon} (z)|+|t^*(y_\varepsilon)|.$$
Let $x^*\in X^*, z_\varepsilon \in Y$ be such that $\|x^*\|=x^*(x)=1$ and $\|z_\varepsilon\|=1$, $y^*(z_\varepsilon)\geq 1-\varpi (\varepsilon)/2$.  Define, $T_{x,y^*,\varepsilon} =x^*_{\varepsilon}(\cdot)z_\varepsilon+x^*(\cdot)y_{\varepsilon}\in \mathcal{F}(X,Y)$. We see that 
\begin{eqnarray*}
y^*(T_{x,y^*,\varepsilon}(x)) - \varpi (\varepsilon)/2 &=&x^*_{\varepsilon} (x) y^*(z_\varepsilon)+  y^*(x_{\varepsilon})- \varpi (\varepsilon)/2\\
                                                                                             &=&  x^*_{\varepsilon} (x) (y^*(z_\varepsilon)-1)+  x^*_{\varepsilon} (x) + y^*(x_{\varepsilon})- \varpi (\varepsilon)/2\\
&\geq& x^*_{\varepsilon} (x) + y^*(x_{\varepsilon}) - \varpi (\varepsilon)\\
&\geq& |x^*_{\varepsilon} (z)|+|t^*(y_\varepsilon)|.
\end{eqnarray*} 
 On the other hand,  
$$|t^*(T_{x,y^*,\varepsilon}(z))|=|x^*_{\varepsilon}(z)t^*(z_\varepsilon)+x^*(z)t^*(y_{\varepsilon})|\leq |x^*_{\varepsilon} (z)| +| t^*(y_{\varepsilon})|.$$ By the previous two inequalities, we obtain 
$$y^*(T_{x,y^*,\varepsilon}(x)) -  \varpi (\varepsilon) /2\geq |t^*(T_{x,y^*,\varepsilon}(z))|,$$ which  contradict $(\ref{R2})$. Hence, $\Gamma(x,y^*) = \Gamma(z,t^*)$  and so $\gamma_\mathcal{P}$ is well defined. Finally, since $(\Gamma(L), d_H)$ is a complete metric space (see Proposition $\ref{P1}$), we see that $(\mathcal{C}_L,\gamma_\mathcal{P})$ is a complete pseudometric space. 
\end{proof}

\begin{lemma} \label{lemC} Let $X, Y, W$  be Banach spaces and $L$ be a symmetric and closed subset of $S_X\times S_{Y^*}$ having property $(P)$. Let $(\mathcal{Z}, \|\cdot\|_\mathcal{Z})$ be a Banach space included in $C_{b,e}(L, W)$, such that:

$(i)$ $\|\cdot\|_\mathcal{Z}\geq \|\cdot\|_{\infty}$.

$(ii)$ $\mathcal{Z}$ contains the space $[\mathcal{F}(X,Y)]\oplus W$ and is such that $\|[T]\oplus w\|_\mathcal{Z}\leq \|T\|\|w\|$ for all $(T,w)\in \mathcal{F}(X,Y)\times W$.

 Then, the set  $(\mathcal{C}_L,\gamma_\mathcal{P})$ has the $w^*USP$ in $\widetilde{\mathcal{Z}}^*$, where $\widetilde{\mathcal{Z}}=I_L(\mathcal{Z})$.
\end{lemma}
\begin{proof}  Since $L$ has property $(P)$ we have:  for every $(x,y^*) \in L$, there exists $(x^*_{\varepsilon},y_{\varepsilon})\in  B_{X^*}\times B_Y$ satisfying: $\forall (z,t^*)\in L$
\begin{eqnarray}  \label{Lab}
	\hspace{5mm}&&\\
	d_H(\Gamma(x,y^*), \Gamma(z,t^*))  \geq \varepsilon \Longrightarrow x^*_{\varepsilon}(x) +y^*(y_\varepsilon)-  \varpi (\varepsilon) \geq  |x^*_{\varepsilon} (z)|+|t^*(y_\varepsilon)|.\nonumber
	\end{eqnarray}

For each $(x,y^*)\in L$ and each $\varepsilon>0$, let $x^*\in X^*, a_\varepsilon \in Y$ be such that $\|x^*\|=x^*(x)=1$ and $\|a_\varepsilon\|=1$, $y^*(a_\varepsilon)\geq 1-\varpi (\varepsilon)/8 $ and define, $T_{x,y^*,\varepsilon} =1/2(x^*_{\varepsilon}(\cdot)a_\varepsilon+x^*(\cdot)y_{\varepsilon})\in \mathcal{F}(X,Y)$. For every $w^*\in S_{W^*}$ and $\varepsilon >0$, choose and fix an $e_{w^*,\varepsilon}\in S_W$ such that 
\begin{eqnarray} \label{Lab1}
1\geq w^*(e_{w^*,\varepsilon}) >1-\lambda,
\end{eqnarray} 
where $\lambda := \frac{\varpi(\varepsilon)}{4(2+\varpi(\varepsilon))}. $ Consider the operator $[T_{x,y^*,\varepsilon}]\oplus e_{w^*,\varepsilon}\in \mathcal{Z}$.  Clearly, $\|I_L([T_{x,y^*,\varepsilon}]\oplus e_{w^*,\varepsilon})\|_{\widetilde{\mathcal{Z}}}=\|[T_{x,y^*,\varepsilon}]\oplus e_{w^*,\varepsilon}\|_\mathcal{Z}\leq 1$. Let $v^* \circ \delta_{\Gamma(z,t^*)} \in \mathcal{C}_L$ such that $$\gamma_\mathcal{P}(w^*\circ \delta_{\Gamma(x,y^*)}, v^*\circ \delta_{\Gamma(z,t^*)}):=d_H(\Gamma(x,y^*), \Gamma(z,t^*)) \geq \varepsilon.$$ On one hand, if we set $A:= v^*\circ \delta_{\Gamma(z,t^*)}(I_L([T_{x,y^*,\varepsilon}]\oplus e_{w^*,\varepsilon}))+\varpi (\varepsilon)/4$, we have  
\begin{eqnarray*}
 A &:=& \frac{x^*_{\varepsilon}(z)t^*(a_\varepsilon)+x^*(z)t^*(y_{\varepsilon})}{2}v^*(e_{w^*,\varepsilon}) +\varpi (\varepsilon)/4\\
                                                                                                                      &\leq& \frac{|x^*_{\varepsilon} (z)| +| t^*(y_{\varepsilon})|}{2} +\varpi (\varepsilon)/4.
\end{eqnarray*}
On the other hand, if we set $B:=w^*\circ \delta_{\Gamma(x,y^*)}(I_L([T_{x, y^*,\varepsilon}]\oplus e_{w^*,\varepsilon}))$, using property $(P)$, we have 
\begin{eqnarray*}
B &:=& \frac{x^*_{\varepsilon}(x)y^*(a_\varepsilon)+y^*(y_{\varepsilon})}{2}w^*(e_{w^*,\varepsilon}) \\
&=&\frac{[x^*_{\varepsilon}(x)+y^*(y_{\varepsilon})] +[x^*_{\varepsilon}(x)(y^*(a_\varepsilon)-1)]}{2}w^*(e_{w^*,\varepsilon})\\
&=& \frac{x^*_{\varepsilon}(x)+y^*(y_{\varepsilon})}{2}w^*(e_{w^*,\varepsilon})+ \frac{x^*_{\varepsilon}(x)(y^*(a_\varepsilon)-1)}{2}w^*(e_{w^*,\varepsilon})\\
&\geq&  \frac{x^*_{\varepsilon}(x)+y^*(y_{\varepsilon})}{2}w^*(e_{w^*,\varepsilon}) -\varpi (\varepsilon)/8\\
                                                                                                                     &\geq& \frac{|x^*_{\varepsilon} (x)| +| y^*(y_{\varepsilon})| + \varpi (\varepsilon)}{2}(1- \lambda)-\varpi (\varepsilon)/8 \hspace{1mm} (\textnormal{Thanks to }(\ref{Lab}) \& (\ref{Lab1}))\\
                                                                                                                       &=& \frac{|x^*_{\varepsilon} (x)| +| y^*(y_{\varepsilon})| +  \varpi (\varepsilon)}{2} -\frac{\lambda}{2}(|x^*_{\varepsilon} (x)| +| y^*(y_{\varepsilon})| + \varpi (\varepsilon))-\varpi (\varepsilon)/8\\
&\geq& \frac{|x^*_{\varepsilon} (x)| +| y^*(y_{\varepsilon})| +  \varpi (\varepsilon)}{2} -\frac{\lambda}{2} (2 + \varpi (\varepsilon))- \varpi (\varepsilon)/8\\
&=& \frac{|x^*_{\varepsilon} (x)| +| y^*(y_{\varepsilon})|}{2}+\varpi (\varepsilon)/4 +[\varpi (\varepsilon)/4 -\frac{\lambda}{2} (2+ \varpi (\varepsilon))-\varpi (\varepsilon)/8]\\
&=& \frac{|x^*_{\varepsilon} (x)| +| y^*(y_{\varepsilon})|}{2}+\varpi (\varepsilon)/4 \hspace{1mm} (\textnormal{Thanks to the definition of } \lambda).
\end{eqnarray*}
Thus, we have $B\geq A$, that is
\begin{eqnarray*}
w^*\circ \delta_{\Gamma(x,y^*)}(I_L([T_{x,y^*,\varepsilon}]\oplus e_{w^*,\varepsilon})) - \varpi (\varepsilon)/4 \geq v^*\circ \delta_{\Gamma(z,t^*)}(I_L([T_{x,y^*,\varepsilon}]\oplus e_{w^*,\varepsilon})).
\end{eqnarray*}
Hence, from Definition ~\ref{USP}, the complete pseudometric space $(\mathcal{C}_L,\gamma_\mathcal{P})$ has the $w^*\mathcal{USP}$ is $\tilde{\mathcal{Z}}^*$. 
\end{proof}

Combining Proposition ~\ref{PP1} applied with $V=X\times Y^*$, Lemma ~\Ref{lemC0} and Lemma ~\Ref{lemC}, we obtain our main result.

\begin{theorem} \label{thmprin} Let $X, Y, W$  be Banach spaces and $L$ be a nonempty symmetric and closed subset of $S_X\times S_{Y^*}$ having the property $(P)$. Let $(\mathcal{Z}, \|\cdot\|_\mathcal{Z})$ be a Banach space included in $(C_{b,e}(L, W), \|\cdot\|_{\infty})$, such that:

$(i)$ $\|\cdot\|_\mathcal{Z}\geq \|\cdot\|_{\infty}$.

$(ii)$ $\mathcal{Z}$ contains the space $[\mathcal{F}(X,Y)]\oplus W$ and is such that $\|[T]\oplus w\|_\mathcal{Z}\leq \|T\|\|w\|$ for all $(T,w)\in \mathcal{F}(X,Y)\times W$.

Then, for every $\Theta \in C_{b,e}(L, W)$ the set 
$$\mathcal{G}(\Theta):=\lbrace \Phi \in \mathcal{Z}: \Theta+\Phi \textnormal{  attains } d_H\textnormal{-strongly its maximum on } L\rbrace,$$
is the complement of a $\sigma$-porous subset of $(\mathcal{Z}, \|\cdot\|_{\mathcal{Z}})$. Moreover, we have the following version of the Bishop–Phelps–Bollobás theorem: given $\varepsilon >0$, there exists $\eta(\varepsilon)>0$ such that whenever $\Phi\in \mathcal{Z}$ with $\|\Phi\|_{\infty}=1$ and $(x,y^*)\in L$, satisfy 
$$\|\Phi(x,y^*)\|>1-\eta(\varepsilon),$$ 
there are $\Psi \in \mathcal{Z}$ with $\|\Psi\|_{\infty}=1$ and $(z,t^*)\in L$ such that 
$$\|\Psi(z,t^*)\|=1, \hspace{1mm} \|x-z\|+\|y^*-t^*\| < \varepsilon \textnormal{ and } \|\Psi-\Phi\|_{\infty}<\varepsilon.$$
\end{theorem}
\vskip5mm
Corollary \ref{Cor1} in the introduction is a particular case of  Theorem \ref{thmprin}  when applied with  $X$ a uniformly convex Banach space and $Y$ a uniformly smooth Banach space and $L=\Pi_\theta (X,Y)$ which has the property $(P)$ thanks to Proposition \ref{corex}, observing in this case that by the definitions, the notion of {\it "attains  $d_H$-strongly its maximum on $\Pi_\theta (X,Y)$"} is equivalent to the notion of  {\it "attains  strongly its $\theta$-numerical radius"}. 

We also obtain the following result.

\begin{corollary} \label{Cor2} Under the hypothesis of Theorem \ref{thmprin}, let $(\Theta_n)\subseteq \mathcal{Z}$ be a sequence. Then, for every $\varepsilon >0$, there exists $\Phi_\varepsilon \in \overline{[\mathcal{F}(X,Y)]\oplus W}^{\|\cdot\|_\mathcal{Z}}\subseteq \mathcal{Z}$ such that $\|\Phi_\varepsilon\|_\mathcal{Z}<\varepsilon$ and $\Theta_n+\Phi_\varepsilon$  attains  $d_H$-strongly its maximum on $ L$ for every $n\in \N$.
\end{corollary}
\begin{proof}  We apply Theorem \ref{thmprin}  to the closed subspace $\mathcal{Z}_1:=\overline{[\mathcal{F}(X,Y)]\oplus W}^{\|\cdot\|_\mathcal{Z}}$ of $(\mathcal{Z}, \|\cdot\|_\mathcal{Z})$. Thus, for each $n\in \N$, the set 
$$\mathcal{G}(\Theta_n):=\lbrace \Phi \in \mathcal{Z}_1: \Theta_n+\Phi \textnormal{  attains } d_H\textnormal{-strongly its maximum on } L\rbrace,$$
is the complement of a $\sigma$-porous subset of $(\mathcal{Z}_1, \|\cdot\|_{\mathcal{Z}})$.  Thus,  $\cap_{n\in \N} \mathcal{G}(\Theta_n)$ is also the complement of a $\sigma$-porous subset of $\mathcal{Z}_1$, in particular it is a dense subset. Hence, for every $\varepsilon >0$, there exists $\Phi_\varepsilon \in \cap_{n\in \N} \mathcal{G}(\Theta_n)\subseteq \overline{[\mathcal{F}(X,Y)]\oplus W}^{\|\cdot\|_\mathcal{Z}}$ such that $\|\Phi_\varepsilon\|_\mathcal{Z}<\varepsilon$ and $\Theta_n+\Phi_\varepsilon$  attains  $d_H$-strongly its maximum on $ L$ for every $n\in \N$.
\end{proof}
The above corollary applies in particular to operators attaining their {\it $\theta$-numerical radius"} in Corollary ~\ref{Cor1}. 

\subsection{Applications to norm attaining operators.} 
Now, we give some applications to the subject of operators attaining their norm. The density of the set of operators attaining their norm in $\mathcal{L}(X,Y)$, under some assymptions on $X$ or $Y$, is a well-known result (see for instance \cite{Aco0} and references therein). We extend below the density by being the complement of $\sigma$-porous set. The first result in this direction (for $\sigma$-porosity) has been proved in \cite[Corollary 4.7]{Ba} when $X$ is a Banach space having  property $(\alpha)$ introduced by Schachermayer in \cite{Sch} (see Corollary \ref{CorA} below for a general version). In  Corollary \ref{CorB} we give a new result concerning the set of operators $T$ whose adjoint $T^*$ attains its norm. This result applies in particular where the space $Y$ has property $(\beta)$ introduced by J. Lindenstrauss \cite{Lj}. 

Recall that an operator $T\in \mathcal{L}(X,Y)$ attains its norm if there exists $x_0\in S_X$ such that $\|T\|=\|T(x_0)\|$.  We say that $T$ attains strongly its norm on a subset $A \subseteq S_X$ at $x_0\in A$, if the map $x\mapsto \|T(x)\|$ attains strongly its maximum on $A$ at $x_0$ (that is, $\|x_n -x_0\|\to 0$, whenever $(x_n)\subseteq A$ and $\|T(x_n)\|\to \|T(x_0)\|$)  and that $\|T\|=\sup \lbrace \|T(x)\|: x \in A \rbrace$.

\begin{corollary} \label{CorA} Let $X$ be a Banach space satisfying the following property: 

$(i)$ there exists $A\subseteq S_X$ such that the absolute convex hull of $A$ is dense in the unit ball of $X$,

$(ii)$ for every $\varepsilon >0$ small enough, there exists $\varpi_X(\varepsilon) >0$ satisfying for every $x\in A$ there exists 
 $x^*_{x,\varepsilon} \in B_{X^*}$ such that  
\begin{eqnarray*}
\forall z\in A: \|z-x\|\geq \varepsilon \Longrightarrow x^*_{x,\varepsilon}(x) -\varpi_X(\varepsilon) \geq |x^*_{x,\varepsilon}(z)|.
\end{eqnarray*}
Then, for every Banach space $W$ and for any closed subspace $\mathcal{A}$ of $\mathcal{L}(X,W)$, containing the space $\mathcal{F}(X,W)$ of finite rank operators, we have: for every $S \in \mathcal{L}(X,W)$ the set 
$$\mathcal{G}(S):=\lbrace T \in \mathcal{A}\hspace{1mm}| \hspace{2mm}  S+T \textnormal{  attains } \textnormal{strongly its norm on } A \rbrace,$$
is the complement of a $\sigma$-porous subset of $\mathcal{A}$. Moreover, we have the following version of the Bishop–Phelps–Bollobás theorem: given $\varepsilon >0$, there exists $\eta(\varepsilon)>0$ such that whenever $T\in \mathcal{A}$ with $\|T\|=1$ and $x\in A$, satisfy 
$$\|T(x)\|>1-\eta(\varepsilon),$$ 
there are $S \in \mathcal{A}$ with $\|S\|=1$ and $z \in A$ such that 
$$\|S(z)\|=1, \hspace{1mm} \|x-z\| < \varepsilon \textnormal{ and } \|S-T\|<\varepsilon.$$

\end{corollary}

\begin{proof} We  apply Theorem \ref{thmprin} with $Y=\R$, the closed symmetric set $$L:=(A\times \lbrace 1\rbrace) \cup - (A\times \lbrace 1\rbrace)\subseteq S_X \times S_{\R^*}=S_X \times \lbrace \pm 1\rbrace,$$ and $$\mathcal{Z}:=\lbrace \tilde{T}=[(x, \lambda)\mapsto \lambda T(x) \in W]: (x, \lambda) \in X\times \R, T\in \mathcal{A} \rbrace \subseteq C_{b,e}(L, W),$$ equipped with the norm 
\begin{eqnarray*} \|\tilde{T}\| &:=&\sup \lbrace \|\tilde{T}(x, \lambda)\|: (x, \lambda)\in L \rbrace \\
&=& \sup \lbrace \|T(x)\|: x \in A \rbrace\\
&=& \|T\| \hspace{1mm} (\textnormal{by using } (i)).
\end{eqnarray*}
We see easily that $[\mathcal{F}(X, \R)]\oplus W \subseteq \mathcal{Z}$, since $\mathcal{F}(X, W)\subseteq \mathcal{B}$ and that $\mathcal{Z}$ is isometrically isomorphic to $\mathcal{A}$.  Notice that, for all $(\xi x, \xi), (\tau z, \tau)\in L$, where $x, z\in A$ and $\xi, \tau \in \lbrace \pm 1\rbrace$, we have that $d_H(\Gamma(\xi x, \xi), \Gamma(\tau z, \tau))=\|x -z\|$. It remains just to see that the set $L$ has the property $P$ which is a trivial consequence of $(ii)$.
\end{proof}
\begin{Exemp} Corollary \ref{CorA}, applies in the following cases:

$(1)$  $X=(l^1,\|\cdot\|)$, with $A=\lbrace e_n: n\in \N \rbrace$, where $(e_n)$ denotes the canonical basis of $l^1$.  

$(2)$ For $1<p<+\infty$, denote $S^+_{l^p}:=\lbrace (x_n)\in S_{l^p}: x_n\geq 0, \forall n \in \N\rbrace$, the intersection of the the unit sphere of $l^p$ with the positive cone. We see using the uniform convexity of $l^p$ that for every $\varepsilon \in (0,2]$, there exists $\delta(\varepsilon) >0$ (the modulus of uniform convexity of $l^p$) satisfying for every $x=(x_n)\in S^+_{l^p}$, there exists $x^*=(x^{p-1}_n)\in B_{l^q}$ ($\frac{1}{q}+\frac{1}{p}=1$) such that $\forall z=(z_n)\in S^+_{l^p}$:
\begin{eqnarray*}
 \|z-x\|_p \geq \varepsilon \Longrightarrow x^*(x)-2\delta(\varepsilon)=1-2\delta(\varepsilon) \geq \sum_{n=0}^{+\infty} x^{p-1}_n z_n =|x^*(z)|.
\end{eqnarray*}
Set $D:=\overline{\textnormal{co}}(\textnormal{bal}(S^+_{l^p}))$, the closed absolute convex hull of $S^+_{l^p}$ and define the Minkowski functional of $D$ which is a norm,
$$N_p(x)=\inf \lbrace \lambda >0: x \in \lambda D\rbrace.$$
 We see easily that $\frac{1}{2} B_{l^p}\subseteq D \subseteq B_{l^p}$, so that $\|\cdot\|_p \leq N_p(\cdot)\leq 2\|\cdot\|_p$, where $\|\cdot\|_p$ denotes the classical norm of $l^p$. Then, Corollary \ref{CorA}, applies for $X=(l^p, N_p)$, $1<p<+\infty$ with $A:=S^+_{l^p}$.

\end{Exemp}

\begin{corollary} \label{CorB} Let $Y$ be a Banach space satisfying the following property: 

$(i)$ there exists $B\subseteq S_{Y^*}$ such that $\|y\|=\sup_{y^*\in B}|y^*(y)|$ for all $y\in Y$,

$(ii)$ for every $\varepsilon >0$ small enough, there exists $\varpi_Y(\varepsilon) >0$ satisfying for every $y^*\in B$ there exists 
 $y_{y^*,\varepsilon} \in B_{Y}$ such that  
\begin{eqnarray*}
\forall t^*\in B: \|t^*- y^*\|\geq \varepsilon \Longrightarrow y^*(y_{y^*,\varepsilon}) -\varpi_Y(\varepsilon) \geq |t^*(y_{y^*,\varepsilon})|.
\end{eqnarray*}
Then, for every Banach space $W$ and for any closed subspace $\mathcal{B}$ of $\mathcal{L}(W,Y)$, containing the space $\mathcal{F}(W, Y)$ of finite rank operators, we have: for every $S \in \mathcal{L}(W,Y)$ the set 
$$\mathcal{G}(S):=\lbrace T \in \mathcal{B}\hspace{1mm}| \hspace{2mm}  S^*+T^*: Y^*\to W^* \textnormal{  attains } \textnormal{strongly its norm on } B \rbrace,$$
is the complement of a $\sigma$-porous subset of $\mathcal{B}$. Moreover, we have the following version of the Bishop–Phelps–Bollobás theorem: given $\varepsilon >0$, there exists $\eta(\varepsilon)>0$ such that whenever $T\in \mathcal{B}$ with $\|T\|=\|T^*\|=1$ and $y^*\in B$, satisfy 
$$\|T^*(y^*)\|>1-\eta(\varepsilon),$$ 
there are $S \in \mathcal{B}$ with $\|S\|=\|S^*\|=1$ and $t^* \in B$ such that 
$$\|S^*(t^*)\|=1, \hspace{1mm} \|t^*-y^*\| < \varepsilon \textnormal{ and } \|S-T\|=\|S^*-T^*\|<\varepsilon.$$

\end{corollary}

\begin{proof} We apply Theorem \ref{thmprin} with $X=\R$, the closed symmetric set $$L:=(\lbrace 1\rbrace \times B) \cup - (\lbrace 1\rbrace \times B)\subseteq S_{\R} \times S_{Y^*}=\lbrace \pm 1\rbrace \times S_{Y^*},$$ and 
\begin{eqnarray*}
\mathcal{Z}&:=&\lbrace \tilde{T}=[(\lambda, y^*)\mapsto \lambda y^*\circ T=\lambda T^*(y^*) \in W^*]: (\lambda, y^*) \in \R \times Y^*, T\in \mathcal{B} \rbrace \\
& \subseteq & C_{b,e}(L, W^*),
\end{eqnarray*}
 equipped with the norm 
\begin{eqnarray*} \|\tilde{T}\| &:=&\sup \lbrace \|\tilde{T}(\lambda,y^*)\|: (\lambda, y^*)\in L \rbrace \\
&=& \sup \lbrace \|T^*(y^*)\|: y^* \in B \rbrace\\
&=& \sup \lbrace \sup_{w\in S_{W}}|T^*(y^*)(w)|: y^* \in B \rbrace\\
&=& \sup_{w\in S_{W}} \sup \lbrace |y^*(T(w))|: y^* \in B \rbrace\\
&=& \|T\| \hspace{1mm} (\textnormal{by using } (i)).
\end{eqnarray*}
We see easily that $[\mathcal{F}(\R, Y)]\oplus W^* \subseteq \mathcal{Z}$, since $\mathcal{F}(W, Y)\subseteq \mathcal{B}$ and that  $\mathcal{Z}$ is isometrically isomorphic to $\mathcal{B}$.  Notice that,  for all $(\xi, \xi y^*), (\tau, \tau t^*)\in L$, where $y^*, t^*\in B$ and $\xi, \tau \in \lbrace \pm 1\rbrace$, we have that $d_H(\Gamma(\xi, \xi y^*), \Gamma(\tau, \tau t^*))=\|t^*- y^*\|$. It remains just to see that the set $L$ has the property $P$ which is a trivial consequence of $(ii)$.
\end{proof}
\begin{Exemp} Corollary \ref{CorB}, applies in particular to $Y=c_0$ or $l^\infty$.
\end{Exemp}
Notice the following well-known fact: a bounded linear operator  $T\in \mathcal{L}(W,Y)$ attains its norm if and only if $T^*\in \mathcal{L}(Y^*,W^*)$ attains its norm and there is $y^*\in S_{Y^*}$ such that $\|T^*y^*\|=\|T^*\|$ with $T^*y^*$ attains its norm at some point $w\in S_W$. In this case, $T$ attains its norm at $w\in S_W$. We can exploit the result in \cite[Theorem 3.15]{JMZ} together with Corollary \ref{CorB} to obtain other examples of the risiduality of operators attaining their norm. 

\subsection*{Acknowledgements}
This research has been conducted within the FP2M federation (CNRS FR 2036) and  SAMM Laboratory of the University Paris Panthéon-Sorbonne.

\bibliographystyle{amsplain}

\end{document}